\documentclass{amsart}
\title[The existence of quasiconformal homeomorphism ...]{The existence of quasiconformal homeomorphism \\between planes with countable marked points}
\author[H. FUJINO]{ {\tiny By} \vspace{1ex} \\  HIROKI\ \ \  FUJINO \vspace{-0ex}}
\address{Graduate~School~of~Mathematics, Nagoya~University, Furo-cho Chikusa-ku Nagoya 464-8602, Japan}
\email{m12040w@math.nagoya-u.ac.jp}
\subjclass[2010]{Primary~51M04, Secondary~51M05.}
\keywords{quasiconformal mappings, extremal distance, and quasicircle.}
\date{\today}

\usepackage[dvips]{graphicx}
\usepackage{amsmath}
\usepackage{amssymb}
\usepackage{amsrefs}
\usepackage{ascmac}
\usepackage{mathrsfs} 
\usepackage{setspace}
\makeatletter    
  \def\@seccntformat#1{\csname the#1\endcsname.\quad}
\makeatother

\allowdisplaybreaks

\newtheorem{thm}{Theorem}[section]
\newtheorem{proposition}[thm]{Proposition}

\newtheorem{lemma}[thm]{Lemma}

\newtheorem{theo}{定理}  
\newtheorem{thma}[theo]{Theorem}  %\renewcommand{\thethma}{\Alph{section}} %chapを変更してうまく合わせる
\newtheorem{coroa}[theo]{Corollary}  %\renewcommand{\thecoroa}{\Alph{section}}

\newtheorem{probemp}{Problem}

\newcommand{\tei}{Teichm\"uller } %Teichmuller
\newcommand{\vai}{V\"ais\"al\"a}    %Vaisala

\newcommand{\edc}[2]{\delta^{\mathbb{C}}\left( #1,#2 \right)}
\newcommand{\edd}[2]{\delta^{D}(#1,#2)}

\newcommand{\zokud}[2]{\mathscr{F}^{D}(#1,#2)}

\newcommand{\mode}{{\rm mod}}
\newcommand{\aut}{{\rm Aut}}

\def\syoumeiowari{\hfill $\Box$}        %四角

\begin{document}

% アブストラクト＋タイトル＝＝＝＝＝＝＝＝＝＝＝＝＝＝＝＝＝＝＝＝＝＝
\begin{abstract}
     
We consider quasiconformal deformations of $\mathbb{C}\setminus\mathbb{Z}$.
We give some criteria for infinitely often punctured planes to be
quasiconformally equivalent to $\mathbb{C}\setminus\mathbb{Z}$.
In particular, we characterize the closed subsets of $\mathbb{R}$
whose compliments are quasiconformally equivalent to $\mathbb{C}\setminus\mathbb{Z}$.     
     
\end{abstract}

\maketitle　%abstractの後に置かないと間隔がおかしくなるようだ。
%ーーーーーーーーーーーーーーーーーーーーーーーーーーーーーー

% 本文  第一節〜＝＝＝＝＝＝＝＝＝＝＝＝＝＝＝＝＝＝＝＝＝＝＝＝＝＝＝
%%%%%%%%%%%%%%%%%%%%%%%%%%%%%%%%%%%%%%%%%%%%%%%%%%%%%%%%%%%%%%%%%%%%%%%%%%%%%%%%%%%%%%%%%%%%%%%%%%%%%%%%%%%%%%%%%%%%%  CHAPTER 1 %%%%%%%%%%%%%%
%\mbox{}\newpage　\thispagestyle{empty}　
\section{Introduction}   %amsartでは自動的に大文字になる
\label{chap1}
% ページ番号設定〜〜〜〜〜〜
\setcounter{page}{1}
\renewcommand{\thepage}{\arabic{page}}
%〜〜〜〜〜〜〜〜〜〜〜〜〜
    
Let $R$ be a Riemann surface.
The \tei  space $T(R)$ is a space which describes all quasiconformal deformations of $R$.
It is well known that $T(R)$ becomes either a finite dimensional complex manifold or a
non-separable infinite dimensional Banach analytic manifold.
$T(R)$ becomes finite dimensional if and only if $R$ is of finite type.
Through the investigation of quasiconformal deformations of a certain infinite type Riemann surface, 
a certain characteristic subspace will be found, which is separable.\vspace{1ex}

%%%%%%%%%%%%%%%%%%%%%%%%%%%%%%%%%%%%%%%%%%%%%%%%%%%%%%%%%%%%%%%%%%%%%%%%%%%%%%%%%%%%%%%%%%%%%%%%%%%%%%%%
The universal \tei space $T(\mathbb{D})$ simultaneously describes all quasiconformal deformations
of all hyperbolic type Riemann surfaces. This arises from the fact that 
each covering $X\rightarrow Y$ induces an embedding of $T(Y)$ into $T(X)$.
On the other hand, $\mathbb{C}\setminus\mathbb{Z}$ covers a certain $n$-punctured 
Riemann sphere for each $n\geq 3$.
Namely $T(\mathbb{C}\setminus\mathbb{Z})$ simultaneously describes all quasiconformal 
deformations of Riemann surfaces of genus $0$ with at least three punctures.
Needless to say, the universal \tei space $T(\mathbb{D})$ also describes them.
However for the reasons mentioned below, 
the \tei space $T(\mathbb{C}\setminus\mathbb{Z})$ is more suitable to describe them than $T(\mathbb{D})$.

For each positive integer $n$, let $R_n=(\mathbb{C}\setminus \mathbb{Z})/\langle z+n\rangle$. 
$R_n$ is an $(n+2)$-punctured Riemann
sphere, and the projection $p_n:\mathbb{C}\setminus\mathbb{Z} \rightarrow R_n$ induces the
embedding $p_n^{\ast}:T(R_n)\hookrightarrow T(\mathbb{C}\setminus \mathbb{Z})$.
The covering transformation group of $p_n$ is the cyclic group $\langle z+n\rangle$, so that,
quasiconformal deformations of $R_n$ correspond to periodic quasiconformal
deformations of $\mathbb{C}\setminus\mathbb{Z}$ with only a period $z+n$.
Then it is shown from McMullen's theorem in \cite{mcmullen1} that $p_n^{\ast}$
is totally geodesic for the \tei metric. By contrast, the embedding of $T(R_n)$ into $T(\mathbb{D})$ is not 
totally geodesic (cf. The Kra--McMullen theorem \cite{mcmullen1}).
Additionally, 
$\mathbb{C}\setminus \mathbb{Z}$ is considered to be one of the smallest Riemann surface which has the above properties, that is,
there exists no Riemann surface except $\mathbb{C}\setminus\mathbb{Z}$ 
which is covered by $\mathbb{C}\setminus\mathbb{Z}$ and covers $R_n$ for all $n$.

Thus, in this paper, we would like to investigate quasiconformal deformations of
$\mathbb{C}\setminus\mathbb{Z}$. In particular, we shall try to find all 
Riemann surfaces which are quasiconformally equivalent to $\mathbb{C}\setminus\mathbb{Z}$.\\

%%%%%%%%%%%%%%%%%%%%%%%%%%%%%%%%%%%%%%%%%%%%%%%%%%%%%%%%%%%%%%%%%%%%%%%%%%%%%%%%%%%%%%%%%%%%
This attempt is reduced to the existence problem of quasiconformal homeomorphism between planes with countable marked points.
In fact, if $R$ is quasiconformally equivalent to $\mathbb{C}\setminus\mathbb{Z}$, 
then $R$ is conformally equivalent to $\mathbb{C}\setminus E$ by a certain closed discrete 
subset $E\subset \mathbb{C}$ (cf. The removable singularity theorem, see \cite[Theorem\ 17.3.]{vaisala2}).
Henceforth, we say that two subsets $E,E'\subset\mathbb{C}$ are quasiconformally equivalent if 
there exists a quasiconformal self-homeomorphism of $\mathbb{C}$ which maps $E$ onto $E'$.
We consider the following problem.
\begin{probemp}
Let $\mathscr{P}$ be the family of all closed discrete infinite subsets $E\subset \mathbb{C}$.
Find all $E \in \mathscr{P}$ which is quasiconformally equivalent to $\mathbb{Z}$.
\end{probemp}

This Problem is analogous to the problem investigated by P. MacManus.
In his paper \cite{macmanus1}, he considered the \textit{usual Cantor--middle--third set} (the \textit{Cantor ternary set}) 
\[
\mathscr{C}=[0,1]\setminus \bigcup_{m=1}^{\infty} \bigcup_{k=0}^{3^{m-1}-1} \left( \frac{3k+1}{3^m},\frac{3k+2}{3^m}\right)
\]
instead of $\mathbb{Z}$. Further he completely characterized subsets which is quasiconformally equivalent to $\mathscr{C}$ 
by several conditions of the Euclidean geometry.\\

%%%%%%%%%%%%%%%%%%%%%%%%%%%%%%%%%%%%%%%%%%%%%%%%%%%%%%%%%%%%%%%%%%%%%%%%%%%%%%%%%%%%%%%%%%%%%%%%%%%%%%%%%
%%%%%%%%%%%%%%%%%%%%%%%%%%%%%%%%%%%%%%%%%%%%%%%%%%%%%%%%%%%%%%%%%%%%%%%%%%%%%%%%%%%%%%%%%%%%%%%%%%%%%%%%%
First, when we take the MacManus proof into consideration, it seems significant to solve our Problem for $E\in \mathscr{P}$
contained in the real line. In this particular case, we obtain the next theorem.

\begin{thma} \label{thmB}
For a monotone increasing sequence $E=\{a_n\}_{n\in \mathbb{Z}} \subset \mathbb{R}$ with 
$a_n \rightarrow \pm \infty$ as $n\rightarrow \pm \infty$, the following conditions are equivalent.
\begin{enumerate}\setlength{\leftskip}{4ex}
\item $E$ is quasiconformally equivalent to $\mathbb{Z}$.
\item There exists a quasiconformal homeomorphism of $\mathbb{C}$ such that $f(n)=a_n$ for all $n\in \mathbb{Z}$.
\item There exists $M\geq 1$ such that the following inequality holds for all $n\in \mathbb{Z},\ k\in \mathbb{N}$;
\[
\frac{1}{M} \leq \frac{a_{n+k}-a_n}{a_n-a_{n-k}} \leq M.
\]
\end{enumerate}
\end{thma}

The last condition derives from the concept of $M$-quasisymmetry. % and the Ahlfors-Beurling extension theorem \cite{ahlfors2}. 
Theorem \ref{lem} (proved in Section \ref{lemma}) shows that if $E\in \mathscr{P}$ lying on the real line is quasiconformally equivalent to $\mathbb{Z}$,
then $E$ can not be bounded from above and below. Therefore the assumption in Theorem \ref{thmB} is necessarily required.
Further Theorem \ref{thmB} completely characterizes the subsets of $\mathbb{R}$ which are quasiconformally equivalent to $\mathbb{Z}$.\\

%%%%%%%%%%%%%%%%%%%%%%%%%%%%%%%%%%%%%%%%%%%%%%%%%%%%%%%%%%%%%%%%%%%%%%%%%%%%%%%%%%%%%%%%%%%%%%%%%%%
%%%%%%%%%%%%%%%%%%%%%%%%%%%%%%%%%%%%%%%%%%%%%%%%%%%%%%%%%%%%%%%%%%%%%%%%%%%%%%%%%%%%%%%%%%%%%%%%%%%
Next, we observe $E\in \mathscr{P}$ whose compliment has an automorphism of infinite order.
In this case, we obtain the next theorem.
\begin{thma} \label{thmA}\ 
Let $E\in \mathscr{P}$ which has the following form;
\[
E=\mathbb{Z}+\{ a_n\}_{n=1}^{m} 
\]
where each $a_n$ satisfies ${\rm Re}(a_n) \in [0,1)$. 

Then, $E$ is quasiconformally equivalent to $\mathbb{Z}$ if and only if $m<+\infty$.
\end{thma}

%%%%%%%%%%%%%%%%%%%%%%%%%%%%%%%%%%%%%%%%%%%%%%%%%%%%%%%%%%%%%%%%%%%%%%%%%%%%%%%%%%%%%%%%%%%%%%%%%%%%%
The assumption for $E$ in Theorem \ref{thmA} means that $\mathbb{C}\setminus E$ has an automorphism
of infinite order $z+1$. On the other hand, if $\mathbb{C}\setminus E$ has an automorphism of infinite
order, we may assume $E$ satisfies the assumption in Theorem \ref{thmA} by composing 
certain Affine transformation. Thus we immediately obtain the following application;\vspace{1ex}

Let $T_0=\bigcup_{n\in \mathbb{N}} p_n^{\ast} \left( T(R_n)\right)$, namely $T_0$ is a subspace
of $T(\mathbb{C}\setminus\mathbb{Z})$ which simultaneously describes all quasiconformal deformations of all Riemann
surfaces of finite type $(0,n)$ with $n\geq 3$. Further, let $T_{\infty}$ be the set of all $[S,f]
\in T(\mathbb{C}\setminus\mathbb{Z})$, the \tei equivalence class of 
the quasiconformal homeomorphism $f:\mathbb{C}\setminus\mathbb{Z} \rightarrow S$, such that there exists 
an automorphism of infinite order in ${\rm Aut}(S)$. Then Theorem \ref{thmA} implies
\begin{coroa} \label{coro1}
\hspace{5ex} $\displaystyle T_{\infty}=\bigcup_{[f]\in {\rm Mod}(\mathbb{C}\setminus\mathbb{Z})} [f]_{\ast}\left(T_0\right)$.
\end{coroa}
Here, ${\rm Mod}(\mathbb{C}\setminus\mathbb{Z})$ is the Teichm\"uller-Modular group of $\mathbb{C}\setminus \mathbb{Z}$.
The subspace $T_0$ is not closed in $T(\mathbb{C}\setminus\mathbb{Z})$. However, it is easily seen that $T_0$ is separable
by its construction.
Further from the McMullen theorem, $T_0$ is geodesically convex with respect to the \tei metrics.

In addition, ${\rm Aut}(S)$ is isomorphic to the stabilizer ${\rm Stab}_{{\rm Mod}(\mathbb{C}\setminus\mathbb{Z})}([S,f])$ 
for $[S,f]\in T(\mathbb{C}\setminus\mathbb{Z})$.
Therefore, the Teichm\"uller-Modular group ${\rm Mod}(\mathbb{C}\setminus\mathbb{Z})$ does not
act properly discontinuously at each point of $\overline{T_{\infty}}$, the closure of $T_{\infty}$.\vspace{1ex}

%%%%%%%%%%%%%%%%%%%%%%%%%%%%%%%%%%%%%%%%%%%%%%%%%%%%%%%%%%%%%%%%%%%%%%%%%%%%%%%%%%%%%%%%%%%%%%%%%%%%%
Note that we can apply the above argument to the another infinite type Riemann surface 
$R'=\mathbb{C}^{\ast}\setminus \{2^n\}_{n\in \mathbb{Z}}$, where $\mathbb{C}^{\ast}=\mathbb{C}\setminus\{0\}$. 
We will discuss this case in detail in Section \ref{another}.\\

\setcounter{theo}{0}

%%%%%%%%%%%%%%%%%%%%%%%%%%%%%%%%%%%%%%%%%%%%%%%%%%%%%%%%%%%%%%%%%%%%%%%%%%%%%%%%%%%%%%%%%%%%%%%%%%%%%%%%%%%%%%%%%%%%%  CHAPTER 2 %%%%%%%%%%%%%%
\section{Preliminaries}

%%%%%%%%%%%%%%%%%%%%%%%%%%%%%%%%%%%%%%%%%%%%%%%%%%%%%%%%%%%%%%%%%%%%
\subsection{Porous sets}

We say that 
a subset $E\subset \mathbb{C}$ is $c$-porous in $\mathbb{C}$ for a constant $c\geq 1$ if
any closed disk $\overline{B}_r(z')$ of radius $r>0$ centered at $z'\in \mathbb{C}$ contains $z$ such that
$B_{r/c}(z)\subset \mathbb{C}\setminus E$.\vspace{1ex}

It is easily seen that;
\begin{itemize}\setlength{\leftskip}{1ex}
\item $\mathbb{Z}+i\mathbb{Z}$ is not porous in $\mathbb{C}$.
\item Any subset of $\mathbb{R}$ is $1$-porous in $\mathbb{C}$, particularly, $\mathbb{Z}$ is $1$-porous in $\mathbb{C}$.
\item $E_1=\mathbb{Z}+i\left\{ 2^n \mid n=0,1,2,\cdots \right\}$ is $8$-porous.
\end{itemize}
\vspace{1ex}

J. V\"ais\"al\"a pointed out that the porosity is preserved by quasiconformal mappings in \cite{vaisala3}. Thus 
it immediately follows that $\mathbb{Z}+i\mathbb{Z}$ is not quasiconformally equivalent to
$\mathbb{Z}$. However, by this way, we cannot decide whether $E_1$ is quasiconformally
equivalent to $\mathbb{Z}$ or not. (Theorem \ref{thmA} proved in Section \ref{sss} shows that $E_1$
is not quasiconformally equivalent to $\mathbb{Z}$.)

%%%%%%%%%%%%%%%%%%%%%%%%%%%%%%%%%%%%%%%%%%%%%%%%%%%%%%%%%%%%%%%%%%%%
\subsection{Quasiconformal mappings and Extremal distances}

Let $D\subset \hat{\mathbb{C}}$ be a domain. For given continua $C_1, C_2\subset D$,
\[
\edd{C_1}{C_2}=\mode(\zokud{C_1}{C_2})
\]
is called the extremal distance between $C_1$ and $C_2$ in $D$, where $\mode$
denotes the $2$-modulus of a curve family and $\zokud{C_1}{C_2}$ denotes the family
of all rectifiable curves which join $C_1$ and $C_2$ in $D$. The definition of $2$-modulus is given by
\[
{\rm mod}(\mathscr{F}) := \inf_{\rho} \int _{\mathbb{C}} \rho(x+iy)^2 dxdy.
\]
where the infimum is taken over all non-negative Borel functions with $\displaystyle \int_{\gamma} \rho |dz|\geq 1$ for all
rectifiable $\gamma \in \mathscr{F}$.

The $2$-modulus coincides with
the reciprocal of the extremal length introduced by L.\ V.\ Ahlfors and A.\ Beurling \cite{beurling2}.
It is well known that a sense preserving homeomorphism $f$ becomes $K$-quasiconformal for a constant $K\geq1$ if
and only if $f$ satisfies the following inequality for any curve family $\mathscr{F}$ in the domain of $f$.
\[
\frac{1}{K}\mode(\mathscr{F})\leq \mode(f(\mathscr{F})) \leq K\mode(\mathscr{F}).
\]

 The next useful lower bound for extremal distances was presented by M.\ Vuorinen in
 \cite[Lemma\ 4.7]{vuorinen3}; For each pair of disjoint continua $C_1, C_2\subset \mathbb{C}$,
 it holds that
\[
\edc{C_1}{C_2} \geq \frac{2}{\pi} \log \left( 1+ \frac{\min_{i=1,2}{\rm diam}(C_i)}{{\rm dist}(C_1,C_2)} \right).
\]
    
%%%%%%%%%%%%%%%%%%%%%%%%%%%%%%%%%%%%%%%%%%%%%%%%%%%%%%%%%%%%%%%%%%%%
\subsection{The Ahlfors three-point condition}

The image of $\dot{\mathbb{R}}=\mathbb{R}\cup \{\infty\}$ under a quasiconformal self-homeomorphism
of the Riemann sphere is called a quasicircle.

A characterization of quasicircles was obtained in \cite{ahlfors2}; For a Jordan curve $C$ in
the Riemann sphere which passes through $\infty$, $C$ is a quasicircle if and only if
there exists $A\geq 1$ such that whenever three distinct points $z_1,z_2,z_3\in C\setminus \{\infty\}$   
lie on $C$ in this order, the following inequality holds.
\[
\frac{|z_1-z_2|}{|z_1-z_3|}\leq A.
\]
This necessary and sufficient condition is called the \textit{three-point condition} 
(or the \textit{bounded turning condition}). 
The necessity of the three-point condition means that if a quasicircle goes far away from a certain point,
it cannot return near this point above a certain rate.
A similar characterization theorem also holds for a Jordan curve which does not pass through $\infty$,
however we will only deal with the former case in this paper.

%%%%%%%%%%%%%%%%%%%%%%%%%%%%%%%%%%%%%%%%%%%%%%%%%%%%%%%%%%%%%%%%%%%%%
\subsection{The Ahlfors--Beurling extension theorem}
An orientation preserving self-homeomorphism $\phi$ of $\mathbb{R}$ is called $M$-quasisymmetric for $M\geq 1$
if the following inequality holds for all $x\in \mathbb{R}$ and all $t>0$.
\[
\frac{1}{M}\leq \frac{\phi(x+t)-\phi(x)}{\phi(x)-\phi(x-t)} \leq M.
\]
We merely say $\phi$ is quasisymmetric if $\phi$ is $M$-quasisymmetric for some $M\geq 1$.

The concept of $M$-quasisymmetry gives a characterization of orientation preserving self-homeomorphisms of real line 
which have a global quasiconformal extension, that is; For a given orientation preserving self-homeomorphism $\phi$
of $\mathbb{R}$, $\phi$ can be extended to a quasiconformal homeomorphism from the upper half plane onto itself if and only if
$\phi$ is quasisymmetric (cf. The Ahlfors--Beurling extension theorem \cite{ahlfors2}).
Moreover it is well known that every quasiconformal self-homeomorphism of the upper half plane is the restriction of a global
quasiconformal homeomorphism. Namely, $\phi$ can be extended to a quasiconformal self-homeomorphism of the Riemann sphere,
if and only if $\phi$ is quasisymmetric.

%%%%%%%%%%%%%%%%%%%%%%%%%%%%%%%%%%%%%%%%%%%%%%%%%%%%%%%%%%%%%%%%%%%%%%%%%%%%%%%%%%%%%%%%%%%%%%%%%%%%%%%%%%%%%%%%%%%%%  CHAPTER 3 %%%%%%%%%%%%%%
\section{Proof of Theorem \ref{thmB}}

%%%%%%%%%%%%%%%%%%%%%%%%%%%%%%%%%%%%%%%%%%%%%%%%%%%%%%%%%%%%%%%%%%%%%%
In this section, we would like to restrict ourselves to $E\in\mathscr{P}$ which lies on $\mathbb{R}$.

\subsection{A criterion} \label{lemma}
We obtain the next criterion for $E\in \mathscr{P}$ contained in $\mathbb{R}$ to be quasiconformally
equivalent to $\mathbb{Z}$.

\begin{thm} \label{lem}
Let $E\in\mathscr{P}$ be contained in $\mathbb{R}$.
If $E$ is quasiconformally equivalent to $\mathbb{Z}$, then $\sup E=+\infty$ and $\inf E=-\infty$.
\end{thm}

\begin{proof}
To obtain a contradiction, assume $\inf E>-\infty$. Then since $E$ is discrete and closed, $\sup E=+\infty$.
Thus numbering $E$ suitably we let $E=\{a_n\}_{n\in \mathbb{N}}$ be a monotone increasing sequence with
$a_n \rightarrow +\infty$ as $n\rightarrow +\infty$.\vspace{1ex}

Let $f:\mathbb{C} \rightarrow \mathbb{C}$ be $K$-quasiconformal mapping with $f(E)=\mathbb{Z}$. Composing
an Affine transformation, we may assume $f(a_1)=0$. For an arbitrary fixed constant $M\geq 1$, consider the set
\[
S:=\left\{\ k\in \mathbb{N}\ \left|\ f(a_k)=\max_{j=1,\cdots,k} f(a_j) \geq M\  \right.\right\}.
\]
Obviously, $S$ consist of infinitely many elements and $\sup S=+\infty$ since $f|_E:E\rightarrow \mathbb{Z}$
is bijective. Therefore it is easily seen that there exists $k\in I$ and exist $\ell,m \in \mathbb{N}$ for $k$ such that
\begin{itemize}\setlength{\leftskip}{8ex}
\item\ \  $k<\ell<m$,
\item\ \  $f(a_{\ell})<0$, and \  $f(a_m)=f(a_k)+1$.
\end{itemize}
\vspace{-2ex}
\begin{figure}[h]
       \begin{center}
           \includegraphics[width=12.5cm]{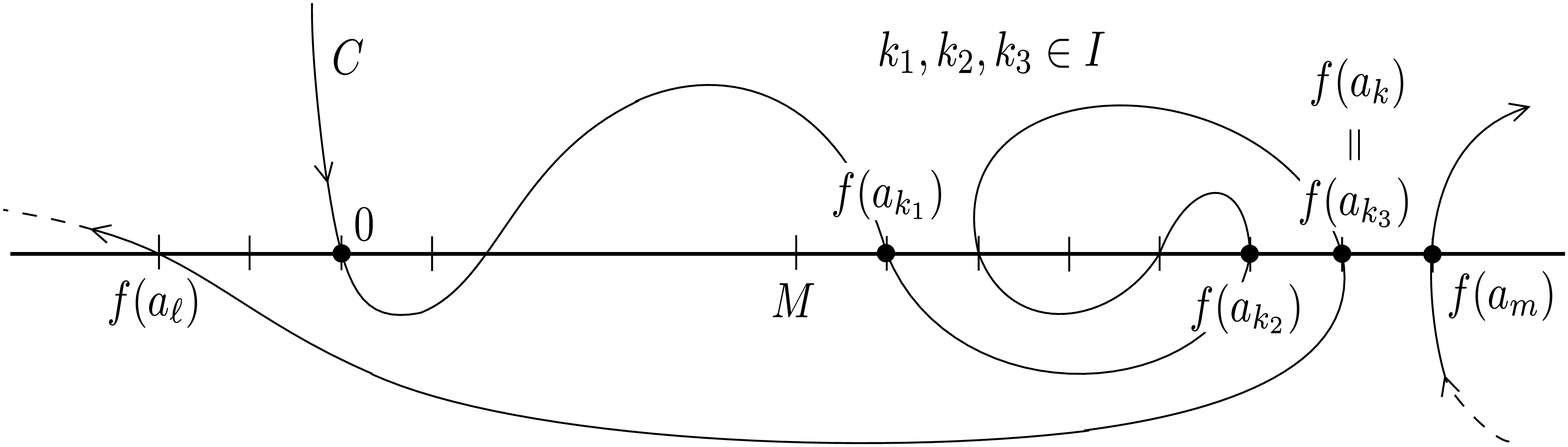} %ファイル名.拡張子を入れる.\ 
       \end{center}
       \caption{}  % \label{hidari}
\end{figure}

\ \vspace{-2ex}\\
Then $f(a_k),f(a_{\ell}),f(a_m)$ lie on $C:=f(\mathbb{R})$ in this order, however,
\begin{eqnarray*}
\frac{|f(a_k)-f(a_{\ell})|}{|f(a_k)-f(a_m)|} =  f(a_k)-f(a_{\ell}) \geq M.
\end{eqnarray*}
\ \\
This means $C$ can not satisfy the three-point condition for any $M\geq1$. However this contradicts that
$C$ is a subarc of a quasicircle which passes through $\infty$.

\end{proof}

This result extremely depends on the particularity of $\mathbb{Z}$. 

\subsection*{Example}
For arbitrary $r,s>1$, $\{r^n\}_{n=0}^{\infty}$ is quasiconformally equivalent to $\{\pm s^n\}_{n=0}^{\infty}$.
     
%%%%%%%%%%%%%%%%%%%%%%%%%%%%%%%%%%%%%%%%%%%%%%%%%%%%%%%%%%%%%%%%%%%%%%
\subsection{A characterization}
From Theorem \ref{lem}, we only need for our Problem to consider the case that 
$E=\{a_n\}_{n\in \mathbb{Z}} \subset \mathbb{R}$ is a monotone increasing sequence
with $a_n \rightarrow \pm \infty$ as $n\rightarrow \pm \infty$.

\begin{thma}
For a monotone increasing sequence $E=\{a_n\}_{n\in \mathbb{Z}} \subset \mathbb{R}$ 
with $a_n \rightarrow \pm \infty$ as $n\rightarrow \pm \infty$, the following conditions are equivalent.
\begin{enumerate}\setlength{\leftskip}{4ex}
\item $E$ is quasiconformally equivalent to $\mathbb{Z}$.
\item There exists a quasiconformal homeomorphism of $\mathbb{C}$ such that $f(n)=a_n$ for all $n\in \mathbb{Z}$.
\item There exists $M\geq 1$ such that the following inequality holds for all $n\in \mathbb{Z},\ k\in \mathbb{N}$;
\[
\frac{1}{M} \leq \frac{a_{n+k}-a_n}{a_n-a_{n-k}} \leq M.
\]
\end{enumerate}
\end{thma}

%%%%%%%%%%%%%%%%%%%%%%%%%%%%%%%%%%%%%%%%%%%%%%%%%%%%%%%%%%%%%%%%%%%%%%%%%%%%%%%%%%%%%%%%%%%%%%%%%%%%%%%%%%%%%%%%%%%%%%%%%%
That (\textit{ii}) implies (\textit{i}) is trivial. We prove the other implications below.

\subsubsection*{Proof ((iii) $\Rightarrow$ (ii))}
Set $\phi (x):=(a_{n+1}-a_n)(x-n)+a_n$ for $x\in [n,n+1)$.
Then $\phi$ defines an orientation preserving self-homeomorphism of $\mathbb{R}$ with $\phi(n)=a_n$, further 
becomes $C(M)$-quasisymmetric where $C(M)=M^4+M^3+M^2+M$. Therefore we obtain a quasicnoformal extension 
$f:\mathbb{C}\rightarrow\mathbb{C}$ of $\phi$ by the Ahlfors--Beurling extension theorem.\vspace{1ex}

Let $x=n+t_1,\ t=m+t_2\ (n\in \mathbb{Z},\ m\in \mathbb{Z}_{\geq 0},\ t_1,t_2 \in [0,1))$.
To prove the quasisymmetry of $\phi$, we have to show the following inequality,
\[
\frac{1}{C(M)} \leq I:=\frac{\phi (x+t)-\phi (x)}{\phi (x)-\phi (x-t)} \leq C(M).
\]
We divide the calculations into the following four cases.

\begin{enumerate}\setlength{\leftskip}{10ex}
\item $t_1+t_2 \in [0,1)$ and $t_1 -t_2 \in (-1,0)$.
\item $t_1+t_2 \in [0,1)$ and $t_1 -t_2 \in [0,1)$.
\item $t_1+t_2 \in [1,2)$ and $t_1 -t_2 \in (-1,0)$.
\item $t_1+t_2 \in [1,2)$ and $t_1 -t_2 \in [0,1)$.
\end{enumerate}
However we only check the first case here as the calculations are almost the same and easy for each case.
To simplify the calculation, we use the next inequality. For $n,m\in \mathbb{Z}\ (n<m)$, and $k\in \mathbb{Z}_{\geq 0}$,
\begin{eqnarray*}
\frac{a_{m+k}-a_n}{a_m-a_n}&\leq & \frac{a_{m+k}-a_{m-1}}{a_m-a_{m-1}}\\
                          &=&\sum_{j=0}^{k} \frac{a_{m+j}-a_{m+j-1}}{a_m-a_{m-1}}\leq M^k+M^{k-1}+\cdots +M+1.
\end{eqnarray*}

Suppose $t_1+t_2 \in [0,1)$ and $t_1 -t_2 \in (-1,0)$. Then
\begin{eqnarray*}
I=\frac{(a_{n+m+1}-a_{n+m})(t_1+t_2)+a_{n+m}-(a_{n+1}-a_n)t_1-a_n}{(a_{n+1}-a_n)t_1+a_n- (a_{n-m}-a_{n-m-1})(1+t_1-t_2)-a_{n-m-1}}.
\end{eqnarray*}
\ \\
(\textit{Upper bound}). First if $m\neq 0$, then we have 
\begin{eqnarray*}
I&< & \frac{(a_{n+m+1}-a_{n+m})+a_{n+m}-a_n}{a_n-(a_{n-m}-a_{n-m-1})-a_{n-m-1}}\\
    &=& \frac{a_{n+m+1}-a_n}{a_n-a_{n-m}} \leq M\frac{a_{n+m+1}-a_n}{a_{n+m}-a_n} \leq M(M+1) <C(M).
\end{eqnarray*}
\ \\
Next if $m=0$, we have
\begin{eqnarray*}
I &= & \frac{(a_{n+1}-a_n)t_2}{(a_{n+1}-a_n)t_1+(a_n-a_{n-1})(t_2-t_1)}\\
  &\leq & M\frac{(a_{n+1}-a_n)t_2}{(a_{n+1}-a_n)t_1+(a_{n+1}-a_n)(t_2-t_1)}=M<C(M).
\end{eqnarray*}
\ \\
(\textit{Lower bound}). First if $m\neq 0,1$, then we have 
\begin{eqnarray*}
I &> & \frac{a_{n+m}-(a_{n+1}-a_n)-a_n}{(a_{n+1}-a_n)+a_n-a_{n-m-1}}\\
 & \geq & \frac{1}{M} \frac{a_{n+m}-a_{n+1}}{a_{n+m+3}-a_{n+1}} \geq \frac{1}{M(M^3+M^2+M+1)} =\frac{1}{C(M)}.
\end{eqnarray*}
\ \\
Next if $m=0$, we have
\begin{eqnarray*}
I &=& \frac{(a_{n+1}-a_n)t_2}{(a_{n+1}-a_n)t_1+(a_n-a_{n-1})(t_2-t_1)}\\
 &\geq & \frac{1}{M} \frac{(a_{n+1}-a_n)t_2}{(a_{n+1}-a_n)t_1+(a_{n+1}-a_n)(t_2-t_1)}=\frac{1}{M} > \frac{1}{C(M)}.
\end{eqnarray*}
\ \\
Finally if $m=1$, we have
\begin{eqnarray*}
I &= & \frac{(a_{n+2}-a_{n+1})(t_1+t_2)+(a_{n+1}-a_n)(1-t_1)}{(a_{n+1}-a_n)t_1+(a_{n-1}-a_{n-2})(t_2-t_1)+(a_n-a_{n-1})}\\
  &\geq & \frac{\frac{1}{M}(a_{n+1}-a_{n})(t_1+t_2)+\frac{1}{M}(a_{n+1}-a_n)(1-t_1)}{M(a_n-a_{n-1})t_1+M(a_n-a_{n-1})(t_2-t_1)+M(a_n-a_{n-1})}\\
  &=& \frac{1}{M^2} \frac{a_{n+1}-a_n}{a_n-a_{n-1}} \geq \frac{1}{M^3} > \frac{1}{C(M)}.
\end{eqnarray*}\\

%%%%%%%%%%%%%%%%%%%%%%%%%%%%%%%%%%%%%%%%%%%%%%%%%%%%%%%%%%%%%%%%%%%%%%%%%%%%%%%%%%%%%%%%%%%%%%%%%%%%%%%%%%%%%%%%%%%%%%%%%%
\subsubsection*{Proof ((i) $\Rightarrow$ (iii))}
Let $f:\mathbb{C}\rightarrow \mathbb{C}$ be $K$-quasiconformal homeomorphism such that
$f(E)=\mathbb{Z}\ (K\geq 1)$. \vspace{1ex}

In this proof, we shall use the following proposition.

\begin{proposition}\label{main}
For a monotone increasing sequence $E=\{a_n\}_{n\in\mathbb{Z}}\subset \mathbb{R}$ such that 
$a_n\rightarrow \pm \infty$ as $n\rightarrow \pm \infty$ and for any $K$-quasiconformal homeomorphism 
$f:\mathbb{C}\rightarrow \mathbb{C}$ which maps $E$ onto $\mathbb{Z}$,
there exists $L\geq 1$ such that the following inequality holds for all $n\in \mathbb{Z}$ and $k\in \mathbb{N}$;
\[
\frac{1}{L} \leq \frac{|f(a_{n+k})-f(a_n)|}{|f(a_n)-f(a_{n-k})|} \leq L.
\]

\end{proposition}\ \vspace{1ex}

Proposition \ref{main} is proved in the next section.\vspace{1ex}

For arbitrary fixed $n\in \mathbb{Z}$ and $k\in \mathbb{N}$,
we set $r=|a_{n+k}-a_n|/|a_n-a_{n-k}|,\ r'=|f(a_{n+k})-f(a_n)|/|f(a_n)-f(a_{n-k})|$ and
$S_1=S^1 (a_n,|a_{n+k}-a_n|),\ S_2=S^1 (a_n,|a_n-a_{n-k}|)$ where $S^1(x,R)$ denotes
the circle of radius $R$ centered at $x$.\vspace{1ex}

 If $r>1$, then by using the Vuorinen theorem, we have
\begin{eqnarray*}
\frac{2\pi}{\log r}= \edc{S_1}{S_2}&\geq &\frac{1}{K} \edc{f(S_1)}{f(S_2)}\\
       &\geq &\frac{1}{K} \frac{2}{\pi} \log \left( 1+
       \frac{\min_{i=1,2} {\rm diam} f(S_i)}{{\rm dist}\left( f(S_1), f(S_2)\right)}\right)\\
       &\geq & \frac{2}{\pi K} \log \left( 1+
       \frac{|f(a_n)-f(a_{n-k})|}{|f(a_{n+k}-f(a_{n-k})|}\right)
       \geq \frac{2}{\pi K} \log \left( 1+\frac{1}{r'+1} \right).
\end{eqnarray*}
If $r<1$, we similarly have
\[
\frac{2\pi}{\log 1/r}\geq \frac{2}{\pi K} \log \left( 1+\frac{1}{1/r'+1} \right).
\]
From Proposition \ref{main}, there exists $L\geq 1$ such that $1/L \leq r' \leq L$
where $L$ does not depend on $n\in \mathbb{Z}$ and $k\in \mathbb{N}$.
Combining the above inequalities, we obtain 
\[
\left( \exp \left( \frac{\pi ^2 K}{\log \left( 1+\frac{1}{L+1}\right)}\right)\right)^{-1}
\leq \frac{|a_{n+k}-a_n|}{|a_n-a_{n-k}|} \leq
\exp \left( \frac{\pi ^2 K}{\log \left( 1+\frac{1}{L+1}\right)}\right).
\]
\syoumeiowari 

%%%%%%%%%%%%%%%%%%%%%%%%%%%%%%%%%%%%%%%%%%%%%%%%%%%%%%%%%%%%%%%%%%%%%%%%%%%%%%%%%%%%%%%%%%%%%%%%%%%%%%%%
%%%%%%%%%%%%%%%%%%%%%%%%%%%%%%%%%%%%%%%%%%%%%%%%%%%%%%%%%%%%%%%%%%%%%%%%%%%%%%%%%%%%%%%%%%%%%%%%%%%%%%%%
\subsection{Proof of Proposition \ref{main}}
We shall prove Proposition \ref{main} from now on.\vspace{1ex}

Under the assumptions of Proposition \ref{main}, we let $C=f(\mathbb{R})$. 
Recall that $C$ is a subarc of a quasicircle which passes through $\infty$. Thus there exists
a constant $A\geq 1$ such that if arbitrary distinct three points $z_1,z_2,z_3$ lie on $C$ in this order, it holds;
\[
\frac{|z_1-z_2|}{|z_1-z_3|}\leq A.
\]
\vspace{1ex}
\begin{lemma} \label{lem2}
For arbitrary $n\in \mathbb{Z}$, it holds that $|f(a_n)-f(a_{n+1})|\leq 2A$.
\end{lemma}

\begin{proof}
Suppose $|f(a_n)-f(a_{n+1})|\geq 2$. Further we may assume $f(a_{n+1})>f(a_n)$ since the same argument mentioned below
can be applied to the case $f(a_n)>f(a_{n+1})$.\vspace{1ex}

%Since $f(a_{n+1})-f(a_n)\geq 2$, there exist at least one integer point between $f(a_{n+1})$ and $f(a_n)$.
It is easily confirmed, there exist $m,\ell \in \mathbb{Z}$ which satisfy the following conditions;
\begin{enumerate}\setlength{\leftskip}{10ex}
\item $m\leq n$ and $n+1 \leq \ell$,
\item $f(a_n)\leq f(a_m) <f(a_{n+1})$ and $f(a_n)< f(a_{\ell}) \leq f(a_{n+1})$,
\item $|f(a_m)-f(a_{\ell})|=1$. (See, Figure \ref{zu}.)
\end{enumerate}
\begin{figure}[h]
       \begin{center}
           \includegraphics[width=11cm]{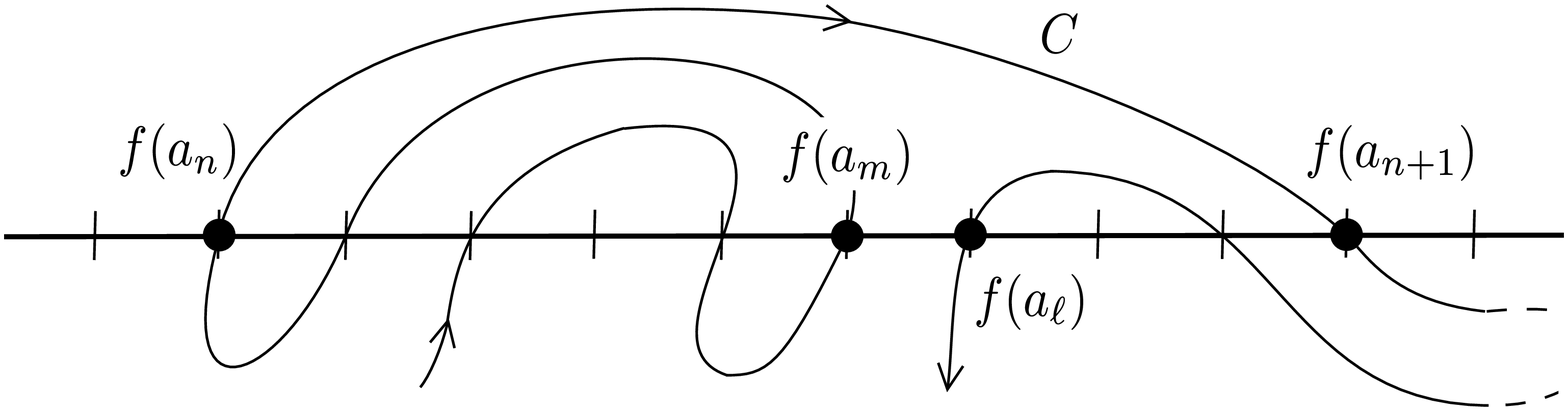} %ファイル名.拡張子を入れる.\ 
       \end{center}
      \caption{}   \label{zu}
\end{figure}\vspace{1ex}

First, suppose $f(a_m)-f(a_n)\geq \left( f(a_{n+1})-f(a_n)\right)/2$. Then $f(a_m),f(a_n),f(a_{\ell})$ are distinct 
since $f(a_m)-f(a_n)\geq 1$, and lie on $C$ in this order. From the three-point condition,
\begin{eqnarray*}
A &\geq & \frac{|f(a_m)-f(a_n)|}{|f(a_m)-f(a_{\ell})|}\\
  &=& f(a_m)-f(a_n) \geq \frac{f(a_{n+1})-f(a_n)}{2}.
\end{eqnarray*}
Thus we have $f(a_{n+1})-f(a_n)\leq 2A$.\vspace{1ex}

Next, suppose $f(a_m)-f(a_n)< \left( f(a_{n+1})-f(a_n)\right)/2$. Then $f(a_{n+1})-f(a_m)>(f(a_{n+1})-f(a_n))/2$ holds.
Since $f(a_{n+1})-f(a_n)>2$,
\[
f(a_{n+1})-f(a_{\ell})> \frac{f(a_{n+1})-f(a_n)}{2}-1 \geq 0,
\]
that is, $\ell \neq n+1$. Therefore $f(a_m),f(a_{n+1}),f(a_{\ell})$ are distinct and are in this order on $C$. Similarly we have
$f(a_{n+1})-f(a_n)\leq 2A$.
\end{proof}
 
\begin{lemma} \label{lem3}
For arbitrary $n\in \mathbb{Z}$ and $k\in \mathbb{N}\ (k\neq 1)$, it follows;
\[
\frac{k-1}{2A} \leq |f(a_n)-f(a_{n+k})| \leq 2Ak.
\]
\end{lemma}

\begin{proof}\ \\
\ \ 
(\textit{Upper bound}). By using the triangle inequality, it immediately follows from Lemma \ref{lem2} that
$|f(a_n)-f(a_{n+k})|\leq 2Ak$.\vspace{1ex}

(\textit{Lower bound}). Suppose $k\neq 1$. The open interval $\displaystyle \left(
f(a_n)-\frac{k-1}{2},\ f(a_n)+\frac{k-1}{2}\right)$ contains at most $(k-1)$ integer points.
Thus there exists $m\in \mathbb{Z}\ (n<m<n+k)$ such that
\[
|f(a_n)-f(a_m)|\geq \frac{k-1}{2}.
\] 
From the three-point condition, we obtain
\[
A \geq \frac{|f(a_n)-f(a_m)|}{|f(a_n)-f(a_{n+k})|} \geq \frac{k-1}{2|f(a_n)-f(a_{n+k})|},
\]
that is, $|f(a_n)-f(a_{n+k})|\geq (k-1)/2A$.
\end{proof} \vspace{1ex}

\subsubsection*{Proof of Proposition \ref{main}}
If $k\neq 1$, it immediately follows from Lemma \ref{lem3} that 
\[
\frac{1}{L} \leq \frac{|f(a_{n+k})-f(a_n)|}{|f(a_n)-f(a_{n-k})|} \leq L
\]
for $L=8A^2$. Moreover, even if $k=1$, it follows from Lemma \ref{lem2}
\[
8A^2>2A\geq \frac{|f(a_{n+1})-f(a_n)|}{|f(a_n)-f(a_{n-1})|} \geq \frac{1}{2A} >\frac{1}{8A^2}.
\]
\syoumeiowari

%%%%%%%%%%%%%%%%%%%%%%%%%%%%%%%%%%%%%%%%%%%%%%%%%%%%%%%%%%%%%%%%%%%%%%%%%%%%%%%%%%%%%%%%%%%%%%%%%%%%%%%%%%%%%%%%%%%%%  CHAPTER 4 %%%%%%%%%%%%%%
\section{Proof of Theorem\ \ref{thmA}}

%%%%%%%%%%%%%%%%%%%%%%%%%%%%%%%%%%%%%%%%%%%%%%%%%%%%%%%%%%%%%%%%%%%%%%
In this section, first, we shall prove Theorem \ref{thmA}. 
Next, we introduce an another example for which almost the same 
result holds. Finally, we would like to suggest a natural question
arising from the above observations.

\subsection{Proof of Theorem \ref{thmA}}
\label{sss}

\begin{thma} \ 
Let $E\in \mathscr{P}$ which has the following form;
\[
E=\mathbb{Z}+\{ a_n\}_{n=1}^{m} 
\]
where each $a_n$ satisfies ${\rm Re}(a_n) \in [0,1)$. 

Then, $E$ is quasiconformally equivalent to $\mathbb{Z}$ if and only if $m<+\infty$.
\end{thma}

\vspace{-1ex}
\begin{proof}\ 

(\textit{Necessity}).
To obtain a contradiction, assume $m=+\infty$.
Let $f:\mathbb{C} \rightarrow \mathbb{C}$ be a $K$-quasiconformal homeomorphism with $f(\mathbb{Z})=E$, 
and by composing an Affine transformation, we may assume $0\in  E$, and  
$\sup \left\{ {\rm Im}(a_n) \mid n\in \mathbb{N}\right\} =\infty$. %\label{upperunboud}

Under the above assumptions, we prove the following lemma.
\vspace{2mm}

\begin{lemma} \label{lemma1}
\hspace{2.5cm}
$\displaystyle 
\sup_{m\in \mathbb{Z}} \left| {\rm Im}f(m) - {\rm Im}f(m+1)\right|=\infty
$.
\end{lemma}

\begin{proof}\ 
Since $\mathbb{Z}$ is porous, by \vai 's theorem, $E$ is $c$-porous for some $c\geq 1$.
For any $r>1$ let $x=i\left\{(\sqrt{2}c+1)r+1\right\}$. Then by porousity of $E$, there
exist $z\in \overline{B}_{\sqrt{2}cr}(x)$ such that $B_{\sqrt{2}r}(z) \subset \mathbb{C}\setminus E$.
Then the square domain $\{w=u+iv \mid  |u-{\rm Re}z|<r,\ |v-{\rm Im}z|<r\}$ does not intersect 
with $E$.
\begin{figure}[h]
       \begin{center}
           \includegraphics[width=12cm]{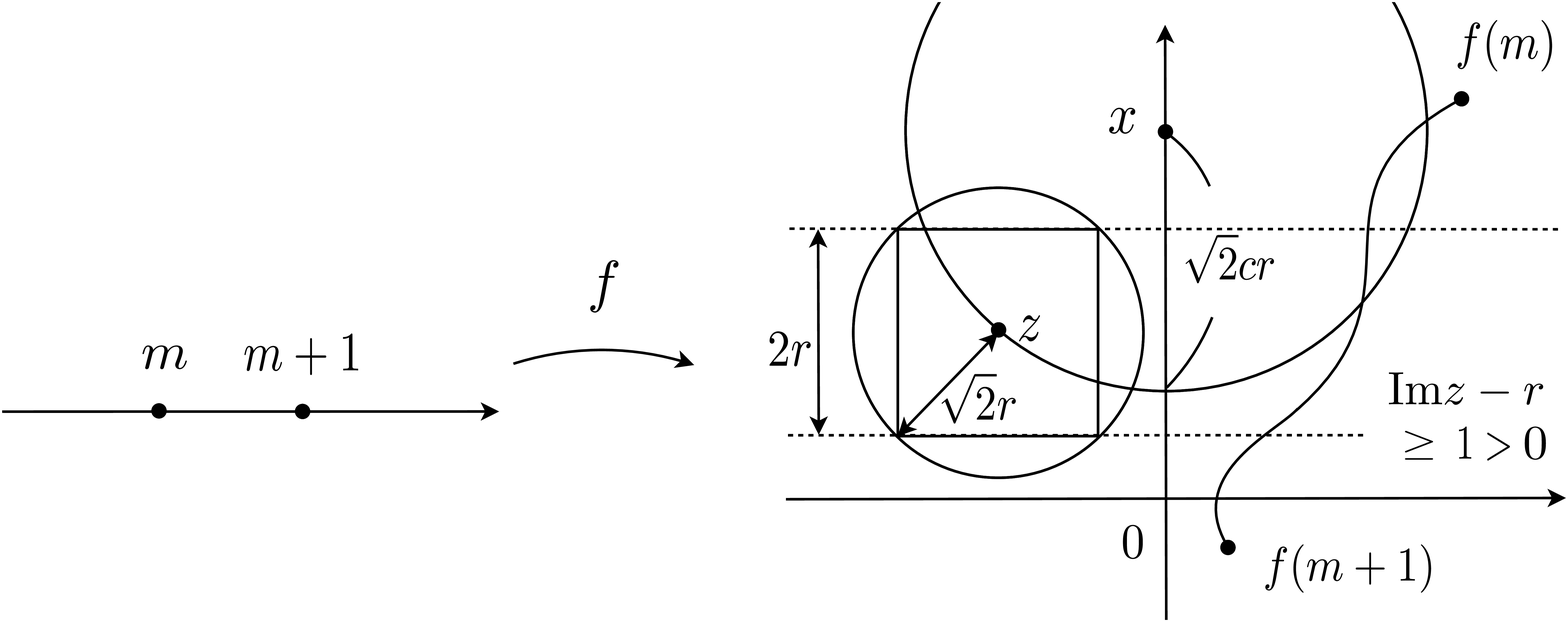} %ファイル名.拡張子を入れる.\ 
       \end{center}
       \vspace{-3ex}\caption{}   %\label{hidari}
\end{figure}\vspace{-0ex}

It is easily confirmed that
\begin{itemize}
\item $E\cap \left\{ w \mid {\rm Im}z-r<{\rm Im}w<{\rm Im}z+r\right\}=\emptyset$, since $z+1\in {\rm Aut}(\mathbb{C}\setminus E)$.
\item $E\cap \left\{ w \mid {\rm Im}w\geq {\rm Im}z+r\right\}\neq \emptyset$, since $\sup \left\{ {\rm Im}a \mid a\in E\right\} =\infty$.
\item $E\cap \left\{ w \mid {\rm Im}z-r\geq{\rm Im}w\right\}\neq\emptyset$, since $0\in E$ and ${\rm Im}z-r\geq 1$.
\end{itemize}

Therefore when we consider the image of real line under $f$, it immediately follows there exists $m\in\mathbb{Z}$
such that $|f(m)-f(m+1)|\geq 2r$.
\end{proof}

\vspace{2mm}
\subsubsection*{Continuation of Proof of Theorem \ref{thm1}}\ 

By Lemma \ref{lemma1}, there exists $m\in \mathbb{Z}$ such that
\[
\ell := |{\rm Im}f(m)-{\rm Im}f(m+1)|>\exp\left( \frac{K\pi^2}{\log 2}\right).
\]
Let
\begin{eqnarray*}
&C_1':=\left\{ f(m)+t \mid t\in [0,1]\right\},\ &C_1:=f^{-1}(C_1')\\
&C_2':=\left\{ f(m+1)+t \mid t\in [0,1]\right\},\ &C_2:=f^{-1}(C_2').
\end{eqnarray*}

Then we have,
\begin{enumerate}
\item by quasiconformality of $f$ \[\edc{C_1}{C_2}\leq K\edc{C_1'}{C_2'} ,\]
\item since $C_1'$ and $C_2'$ are separeted by the annulus $\{ z \mid 1<|z-f(m)|<\ell\}$
\[\edc{C_1'}{C_2'}\leq \frac{2\pi}{\log \ell}<\frac{2}{K\pi}\log 2,\]
\item From Vuorinen's theorem,
\[
\edc{C_1}{C_2} \geq \frac{2}{\pi} \log \left( 1+ \frac{\min_{i=1,2}{\rm diam}(C_i)}{{\rm dist}(C_1,C_2)} \right)\geq \frac{2}{\pi} \log \left( 1+ \min_{i=1,2} {\rm diam} C_i\right).\]
\end{enumerate}

\begin{figure}[h]
       \begin{center}
           \includegraphics[width=11cm]{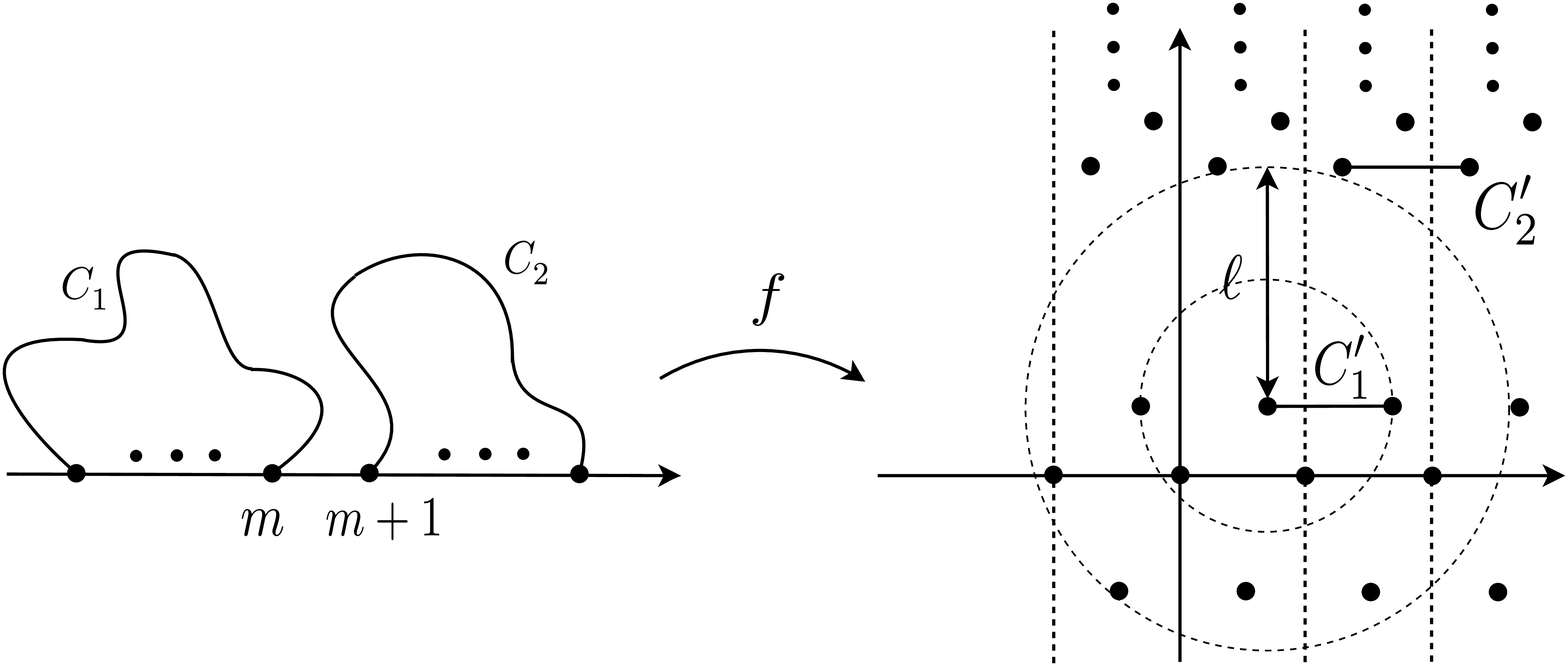} %ファイル名.拡張子を入れる.\ 
       \end{center}
       \caption{}   %\label{hidari}
\end{figure}

Combining the above inequalities, we obtain
\[
\min_{i=1,2} {\rm diam}C_i <1.
\]
On the other hand, since each endpoints of $C_i$ are in the integer set, ${\rm diam}C_i\geq 1\ (i=1,2)$. This is a contradiction. \\

(\textit{Sufficiency}).
Since $\left(\mathbb{C}\setminus \mathbb{Z}\right)\big/ \langle z+m \rangle$ and
$\left( \mathbb{C}\setminus E\right) \big/ \langle z+1\rangle$ are $(m+2)$-punctured Riemann sphere,
there exists a quasiconformal homeomorphism between them which fixes $0$ and $\infty$. Then it 
can be lifted to a quasiconfromal homeomorphism between $\mathbb{C}\setminus \mathbb{Z}$ 
and $\mathbb{C}\setminus E$  
\end{proof}
%\hfill $\Box$
\subsection*{Remark}
The necessity and the proof of the sufficiency part shows that if a Riemann surface $R$ which has an automorphism of
infinite order is quasiconformally
equivalent to $\mathbb{C}\setminus \mathbb{Z}$,
then there exists a periodic quasiconformal deformation from $\mathbb{C} \setminus \mathbb{Z}$ to $R$ coming from the deformation of finitely
punctured Riemann sphere. 

\begin{coroa}
\hspace{5ex} $\displaystyle T_{\infty}=\bigcup_{[f]\in {\rm Mod}(\mathbb{C}\setminus\mathbb{Z})} [f]_{\ast}\left(T_0\right)$.
\end{coroa} 

Here, symbols used in Corollary \ref{coro1} are defined in Introduction.

%%%%%%%%%%%%%%%%%%%%%%%%%%%%%%%%%%%%%%%%%%%%%%%%%%%%%%%%%%%%%%%%%%%%%%
\subsection{Another example} \label{another}
%\bigskip
For a Riemann surface $R$, let ${\rm Aut}_{\infty}(R)\subset \aut(R)$ be the set of all automorphism of infinite order.%=\left\{h\in {\rm Aut}(R) \mid {\rm ord}(h)=\infty \right\}$ 
The following Theorem \ref{thm1} is a mere rephrasing of Theorem \ref{thmA}.

\begin{thm} \label{thm1}
For $E\in \mathscr{P}$ with ${\rm Aut}_{\infty}(\mathbb{C}\setminus E)\neq \emptyset$,
the following are equivalent.

\begin{enumerate}\setlength{\leftskip}{4ex}
\item $\mathbb{C}\setminus E$ is quasiconformally equivalent to $\mathbb{C}\setminus\mathbb{Z}$.
\item For any $h\in {\rm Aut}_{\infty}(\mathbb{C}\setminus E)$, the quotient space $\left( \mathbb{C}\setminus E\right)/ \langle h\rangle $
      is a finitely punctured Riemann sphere.
\item There exists $h\in {\rm Aut}_{\infty}(\mathbb{C}\setminus E)$ such that the quotient space $\left( \mathbb{C}\setminus E\right)/ \langle h\rangle $
      is a finitely punctured Riemann sphere.
\end{enumerate}
\end{thm}

Now, we would like to consider the another infinite type Riemann surface
$R'=\mathbb{C}^{\ast}\setminus \{2^n\}_{n\in \mathbb{Z}}$, where $\mathbb{C}^{\ast}=\mathbb{C}\setminus\{0\}$. 
In this case, a similar theorem is proved far more easily than the case of $\mathbb{C}\setminus\mathbb{Z}$, because of
the relative compactness of the fundamental domain of $\langle 2^nz\rangle$ % in $\mathbb{C}^{\ast}$ 
(contrary to this, the fundamental domain of $\langle z+n\rangle$ is not relatively compact in $\mathbb{C}$).

\begin{thm} \label{thm2}
For a closed discrete infinite subset $E\subset \mathbb{C}^{\ast}$ with ${\rm Aut}_{\infty}(\mathbb{C}^{\ast} \setminus E)\neq \emptyset$,
the following are equivalent.

\begin{enumerate}\setlength{\leftskip}{4ex}
\item $\mathbb{C}^{\ast} \setminus E$ is quasiconformally equivalent to $R'$.
\item For any $h\in {\rm Aut}_{\infty}(\mathbb{C}^{\ast}\setminus E)$, the quotient space $\left( \mathbb{C}^{\ast} \setminus E\right)/ \langle h\rangle $
      is a finitely punctured torus.
\item There exists $h\in {\rm Aut}_{\infty}(\mathbb{C}^{\ast}\setminus E)$ such that the quotient space $\left( \mathbb{C}^{\ast}\setminus E\right)/ \langle h\rangle $
      is a finitely punctured torus.
\end{enumerate}
\end{thm}

Moreover, a theorem similar to Corollary \ref{coro1} also holds.
In this case, the space corresponding to $T_0$ simultaneously describes
all quasiconformal deformations of all Riemann surfaces of finite type
$(1,n)$ with $n\geq 1$, and has the same properties of $T_0$, separability and geodesic convexity.

%%%%%%%%%%%%%%%%%%%%%%%%%%%%%%%%%%%%%%%%%%%%%%%%%%%%%%%%%%%%%%%%%%%%%%%%%%%%%%%%%%%%%%%%%%%%%%%%%%%%%%%%%%%%%%%%%%%%%%%%%%%%%%%%%%%%5
\subsection{Natural question}
With the observations mentioned above, a natural question arises; Does an analogous theorem %to Theorem \ref{thm1} and \ref{thm2} 
hold for Riemann surfaces which have the following property?
\begin{itemize}
\item It has an automorphism of infinite order.
\item For any automorphism of infinite order, the quotient space by the action of its cyclic group is of finite type.
\end{itemize} 
Namely, is the above property preserved by quasiconformal deformations?

For example, the Riemann surface defined by $\displaystyle w^2= z\prod_{n=1}^{\infty}\left(1-\frac{z^2}{n^2}\right)$ has the
above property.

%ーーーーーーーーーーーーーーーーーーーーーーーーーーーーーーー

% 参考文献＝＝＝＝＝＝＝＝＝＝＝＝＝＝＝＝＝＝＝＝＝＝＝＝＝＝＝＝＝＝

%ーーーーーーーーーーーーーーーーーーーーーーーーーーーーーーー

\end{document}